\documentclass{amsart}
\usepackage{amssymb}
\usepackage{comment}
\usepackage{soul}
\usepackage{mathrsfs}
\usepackage[stretch=10,shrink=10]{microtype}
\usepackage[showframe=false, margin = 1.5 in]{geometry}
\usepackage{mathtools}
\usepackage{thmtools} 
\usepackage{thm-restate}
\mathtoolsset{showonlyrefs} 
\usepackage[shortlabels]{enumitem}
\usepackage{ifthen}

\usepackage{stmaryrd}
\usepackage{tikz,tikz-cd}
\usepackage{array}
\usepackage{color}
\usepackage{colortbl}
\usepackage{hyperref}
\hypersetup{
     colorlinks   = true,
     citecolor    = black
}
\usepackage{theoremref}
\usepackage[labelfont=bf]{caption}
\usepackage{subcaption}
\newcommand{\bbP}{\mathbb{P}}
\newcommand{\bbN}{\mathbb{N}}
\newcommand{\bbZ}{\mathbb{Z}}
\newcommand{\bbR}{\mathbb{R}}

\newtheorem{thm}{Theorem}[section]
\newtheorem{lemma}[thm]{Lemma}

\newtheorem{cor}[thm]{Corollary}

\newtheorem*{conjecture*}{Density Conjecture}

\theoremstyle{remark}

\theoremstyle{definition}

\title{Coupling Invasion and First Passage Percolation}
\author{Aldo Morelli}
\address{Department of Mathematics and Statistics\\
University of New Mexico}
\email{amorelli@unm.edu}

\begin{document}
    \maketitle
    \vspace{-2em}
 \begin{abstract}
    It is well known that a continuous phase transition in Bernoulli bond percolation on the integer lattice is equivalent to a vanishing probability a vertex is invaded in invasion percolation. We provide a coupling between invasion percolation and first passage percolation with log-uniform passage times. This yields a new equivalent condition for a continuous phase transition in bond percolation.
\end{abstract}

\section{Introduction}
      The study of percolation models began in 1957 with Broadbent and Hammersley \cite{BH57}, who introduced, for directed crystal lattices, both Bernoulli bond percolation\textemdash a process of forming a random subgraph %
     where each edge is independently included according to a Bernoulli random variable\textemdash and its associated critical threshold. 
     This model has been used to study a variety of phenomena, including infectious disease, cyber security, and quantum magnetism \cite{ML07,CJLP24, ZYH24}. A significant open problem in bond percolation is the existence of an infinite connected component at the critical threshold on the lattice $\bbZ^d$ for dimensions $2<d<11$.   
    
    The model of invasion percolation was introduced in 1983 by Wilkinson and Willemsen \cite{WW83}, inspired by the work of Chandler, Koplik, Lerman, and Willemsen \cite{CKLW82} in the previous year on the flow of fluids in porous media. Invasion percolation is a random process where edges in a graph are independently given a weight uniformly from 0 to 1 and non-invaded edges and vertices are sequentially invaded according to the minimal weight of a non-invaded edge adjacent to any invaded edge. Later, Chayes, Chayes, and Newman \cite{CCN87} utilized invasion percolation to study Bernoulli bond percolation. In so doing, they provided an equivalence between the density of the invaded set and the existence of an infinite connected component at the critical threshold of bond percolation. In 2015, Damron, Lam, and Wang \cite{DLW15} applied invasion percolation to study the properties of two dimensional first passage percolation when edges are given zero passage time with probability equal to the critical threshold of two dimensional bond percolation. 
    
    We provide a coupling between invasion percolation and first passage percolation where passage times are distributed according to a log-uniform distribution, with the hope that this coupling will allow for the study of invasion percolation utilizing the techniques of first passage percolation. When coupled, the invaded vertices within the ball of radius $r$ about the origin are identical to those with passage time less than the passage time to the boundary of the ball with radius $R$ about the origin with probability $1-\epsilon$ for large $R\gg r$, see Theorem \ref{thm:coupling} for a precise statement. Corollary \ref{cor:FPPandBondPerc} follows from this theorem, relating the continuity of the first order phase transition of Bernoulli bond percolation to a condition on vertices in log-uniform FPP having passage time less than the boundary of the ball of radius $R$. We also provide simulations of the first passage process with log-uniform distribution on $\bbZ^2$, giving some evidence that the probability a vertex within the ball of radius $R$ has passage time greater than the passage time to the boundary of the ball approximately follows a power law in $|x|$ for fixed $R$. 
    
    \section{Preliminaries}
    Consider the $d$-dimensional lattice $\bbZ^d$ with nearest neighbor edge set $E$ and origin $0$, endowed with the geodesic distance
    \begin{align*}
        d(v, u):=\min \{|\gamma|:\gamma\text{ is a path from $v$ to $u$}\}.
    \end{align*} 
    For $R\in\bbN$, let $B_R=\{v\in \bbZ^d: d(0, v)\le R\}$, $E_R$ denote the set of edges between vertices in $B_R$, and $\partial B_R$ denote the boundary of $B_R$, i.e.\ the set of vertices of distance exactly $R$ from the origin. 

    \subsection{Bond Percolation}
    Assign independently to each edge in $E$ a \emph{weight} uniformly from $[0, 1]$, and let $w(e)$ denote the weight of the edge $e$. Following \cite{DC18}, we define a \emph{configuration} $\sigma =(w(e): e\in E)$ to be an element of $[0,1]^{E}$. We consider the probability space $\Omega = ([0,1]^{E}, \mathcal{F}, \bbP)$ where $[0,1]^{E}$ is the set of edge weight assignments to $E$, $\mathcal{F}$ is the $\sigma$-algebra generated by events depending on finitely many edges, and $\bbP$ is the corresponding product measure, where each coordinate is a uniformly distributed random variable on $[0,1]$.

    Given a parameter $p$, define Bernoulli bond percolation to be the random subgraph of $\bbZ^d$ consisting of all vertices as well as the edges $e\in E$ such that $w(e)<p$. For $p\in[0,1]$, define
    \begin{align*}
        \theta_d(p):=\bbP[0 \text{ is in an infinite cluster of bond percolation}]
    \end{align*}
    and 
    \begin{align*}
        p_{c,d} := \inf\{p:\theta_d(p)>0\}.
    \end{align*} 
    $\theta_d(p)$ is known to be continuous on $[0, p_{c,d})\cup (p_{c,d},1]$ for all $d$ and right continuous at $p_{c,d}$. To show the continuity of $\theta_d(p)$, it thus suffices to show that $\theta_d(p_{c,d})=0$. For $2<d<11$, this problem is open. 

    \subsection{Invasion Percolation} 
    Utilizing the same probability space $\Omega$ as with bond percolation, define \emph{invasion percolation} (IP) for configuration $\sigma\in\Omega$ to be a process of invading vertices and edges, beginning with $0$, and at each subsequent step invading the edge (as well as its incident vertices) with the smallest weight among the non-invaded edges neighboring already invaded vertices. Because $E$ is countable, almost surely all assigned weights are distinct, and we will restrict ourselves to this case in what follows so that there will never be any ambiguity as to which edge to invade. This process produces a corresponding sequence of invaded vertices, where we discard any repeat vertices, leaving at most one newly invaded vertex in each step. As such, we see that any vertex in $B_R\setminus \partial B_R$ will be invaded in the same step as some edge in $E_R$. We assign a total ordering on the invaded vertices and edges by associating an invaded vertex/edge to the step in which it was invaded and comparing these steps, and denote this order with $<_{\texttt{IP}}$. When comparing vertices and edges which were invaded in the same step, we choose the convention of edges being invaded before vertices.  
    
    Denoting $\mathscr{I}(x)$ to be the event that a vertex $x\in \bbZ^d$ is invaded, in \cite{CCN87} Chayes, Chayes, and Newman showed that for all $d$,
    \begin{equation}\label{eqn:invasion_condition}
     \inf_{x\in \bbZ^d}\bbP[\mathscr{I}(x)]=0\iff \theta_d(p_{c,d})=0,
    \end{equation}
    giving a condition for the continuity of $\theta_d$. 

    \subsection{First Passage Percolation} 
    Following \cite{ADH16}, first passage percolation (FPP) is a stochastic model which assigns a random metric to the lattice $\bbZ^d$. Within this model, each edge in $E$ is given a \emph{passage time} i.i.d.\ according to some distribution $F$. Throughout this paper we will be concerned only with $F=\text{Law}(e^{KT})$, where $T\sim U(0,1)$ and $K>0$ is fixed, and so let $\tau_{K, e}$ denote the passage time of the edge $e$ sampled from the distribution of $e^{KT}$.
    The passage time of a path $\gamma$ is defined by
    \begin{align*}
        T_K(\gamma) := \sum_{e\in \gamma} \tau_{K, e},
    \end{align*}
    and a metric on $E$ is given by 
    \begin{align*}
        T_K(x, y):= \inf_{\gamma} T_K(\gamma)
    \end{align*}
    for $x,y\in \bbZ^d$ and paths $\gamma$ from $x$ to $y$. %
    $T_K(x,y)$ is referred to as the \emph{passage time} between $x$ and $y$. Because $F$ is continuous, 
    all passage times are almost surely distinct, and almost surely there exists a minimizing path for each pair of vertices in $\bbZ^d$. Going forward we restrict ourselves only to such configurations. 
    
    For fixed $R, K\ge 0$, we define the passage time to $\partial B_R$ by
    \begin{align*}
        T_K(0, \partial B_R):= \min_{v\in \partial B_R} T_K(0, v),
    \end{align*}
    and define the passage time from the origin to an edge $e\in E$ by 
    \begin{align*}
        T_K(0, e):=\min(T_K(0, v_1), T_K(0, v_2))+\tau_{K,e},
    \end{align*}
    where $v_1, v_2$ are the vertices incident to $e$. For a vertex $x\in B_R\setminus\partial B_R$ and path $\gamma$ from $0$ to $x$ such that $T_K(\gamma)=T_K(0, x)$, by the minimality of $\gamma$, 
    \begin{equation}\label{eqn:vertextoedge}
        T_K(0, x)=T_K(0, e)
    \end{equation}
    for the final edge $e\in E_R$ of $\gamma$. We assign a total ordering to the edges in $E_R$ such that $e<_{\texttt{FPP}}e'$ whenever $T_K(0, e)<T_K(0, e')$. %
    
    \section{Coupling and Main Theorem} In this section, we couple invasion percolation and FPP with passage time distribution $e^{KT}$ by letting $\tau_{K, e}=e^{K w(e)}$ for a configuration $\sigma\in\Omega$. For large $K$, the behavior of the two coupled models is often identical for vertices close to the origin. More specifically, we say that, when coupled on a configuration $\sigma\in\Omega$, IP \emph{contains} $(K, R)$ log-uniform FPP on a set of vertices $V$ if any vertex $v\in V$ such that $T_K(0, v)<T_K(0, \partial B_R)$ is invaded. Similarly we say that $(K, R)$ log-uniform FPP \emph{contains} IP on $V$ if any vertex $v\in V$ which is invaded also is such that $T_K(0, v)<T_K(0, \partial B_R)$. Letting 
    \[\delta(R, \epsilon)=\frac{1-(1-\epsilon)^{\frac{1}{|E_R|}}}{|E_R|-1},\]
    we have the following:

    \begin{thm}\label{thm:coupling} %
        For any $\epsilon >0$, $r\ge 0$, there exists an $R_0=R_0(\epsilon, r)$ such that for all $R\ge R_0$  
        \begin{equation}\label{eqn:IPcontainsFPP}
             \bbP[\text{IP contains }(K,R)\text{ log-uniform FPP on } B_R] \ge 1-\epsilon,
        \end{equation}
        and
        \begin{equation}\label{eqn:FPPcontainsIP}
            \bbP[(K, R)\text{ log-uniform FPP contains IP on } B_r] \ge 1-\epsilon
        \end{equation}
        where $K=K(R, \epsilon/2):=\tfrac{\log|E_R|}{\delta(R, \epsilon/2)}=O(\epsilon^{-1}R^{4d}\log R )$ as $R\to\infty, \epsilon\to 0^{+}$.
    \end{thm}

    Theorem \ref{thm:coupling} gives us the following corollary regarding $\theta_d(p_{c,d})$. 
    \begin{cor}\label{cor:FPPandBondPerc}
    For $K=K(R, \epsilon/2)$ as in Theorem \ref{thm:coupling},
        \begin{align*}
            \inf_{x\in\bbZ^d}\liminf_{R\to\infty}\bbP[T_{K}(0, x)<T_K(0, \partial B_R)]=0\iff \theta_d(p_{c,d})=0.
        \end{align*}
    \end{cor}

    Through the coupling of Theorem \ref{thm:coupling}, by taking $R$ and $K$ large enough, we may guarantee with arbitrarily high probability that the behavior of the invasion percolation and short term log-uniform FPP models at any specific vertex are identical, in the sense that a vertex is invaded exactly when its passage time is less than that of $\partial B_R$. In such cases, we may replace the condition on vertices being invaded in equation \eqref{eqn:invasion_condition} with the condition that vertices have passage time less than $\partial B_R$, leading to Corollary \ref{cor:FPPandBondPerc}.  

    \subsection{Proof Sketch} By first establishing that on the event that no two weights of edges inside $E_R$ are within some small $\delta$ of one another, the orders $<_{\texttt{IP}}$ and $<_{\texttt{FPP}}$ are identical up until the first invaded vertex in $\partial B_R$, we immediately see that with high probability any vertex which has passage time less than $T_K(0, \partial B_R)$ must have been invaded, giving \eqref{eqn:IPcontainsFPP}. Additionally, because the probability that any vertex close to the origin is invaded after a vertex in $\partial B_R$ may be made arbitrarily small by taking $R$ large enough, the reverse holds with high probability close to the origin, giving \eqref{eqn:FPPcontainsIP}.

    \section{Proof of Theorem \ref{thm:coupling}}
    
    Before giving the proof of Theorem \ref{thm:coupling}, we state the following lemmas, which are proven in Section \ref{sec:proofs_of_lemmas}.

    \begin{restatable}{restatelemma}{backbone}
    \label{lem:Backbone}
    Let $v^{\texttt{IP}}_R$ denote the first vertex in $\partial B_R$ to be invaded and let $\mathscr{T}_{\delta}$ denote the event that no two edge weights in $E_R$ are within $\delta$ of one another. With $K(R, \epsilon)$ as in Theorem \ref{thm:coupling}, on the event $\mathscr{T}_{\delta}$, for all edges $e, e'$ such that $e, e'<_{\texttt{IP}} v_R^{\texttt{IP}}$, $e<_{\texttt{IP}} e'$ if and only if $e<_{\texttt{FPP}} e'$.
    \end{restatable}

    \begin{restatable}{restatelemma}{nolatesmallinvasions}
    \label{lem:NoLateSmallInvasions}
        Let $\mathscr{B}_{r,R}$ denote the event that there is a vertex in $B_r$ which is invaded after $v^{\texttt{IP}}_R$. For any $\epsilon>0$ and $r\ge 0$ there exists $R_0(r, \epsilon)\in\bbN$ such that for all $R\ge R_0$, 
        \begin{align*}
            \bbP[\mathscr{B}_{r,R}]<\epsilon.
        \end{align*}
    \end{restatable}

    \begin{proof}[Proof of Theorem \ref{thm:coupling}] 
     
    Observe that, for
    \begin{equation}\label{eqn:delta}
     \delta(R, \epsilon)=\frac{1-(1-\epsilon)^{\frac{1}{|E_R|}}}{|E_R|-1} 
    \end{equation}
    $\bbP[\mathscr{T}_{\delta}]=1-\epsilon$.
    This follows from calculating the probability that the minimum distance between $|E_R|$ independent uniform random variables on $[0,1]$ is at least $\delta$ \cite{E22}. 
    
    Let $\epsilon>0, r\ge 0$ be given. 
    By Lemma \ref{lem:Backbone} and \eqref{eqn:vertextoedge}, on the event $\mathscr{T}_{\delta}$, any vertex $v\in B_R$ with $T_{K}(O, v)<T_K(0, \partial B_R)$ must have been invaded, and hence 
    \begin{align*}
        \bbP[\text{IP contains }(K,R)\text{ log-uniform FPP on } B_R] &\ge \bbP[\mathscr{T}_{\delta(R, \epsilon)}]
        \\
        &= 1-\epsilon. 
    \end{align*}
    Similarly, by Lemma \ref{lem:Backbone}, on the event $\mathscr{T}_{\delta}$, for there to be an invaded vertex $v\in B_r$ with $T_{K}(O, v)\ge T_K(0, \partial B_R)$, $v$ must have been invaded after $v_R^{\texttt{IP}}$. By Lemma $\ref{lem:NoLateSmallInvasions}$, let $R_0$ be large enough so that for $R\ge R_0$ the probability that there is a vertex in $B_r$ invaded after $v^{\texttt{IP}}_R$ is at most $\epsilon/2$, and so for $R\ge R_0$, 
    \begin{align*}
        \bbP[(K, R)\text{ log-uniform FPP contains IP on } B_r] &\ge 1-\bbP[\mathscr{T}_{\delta(R, \epsilon/2)}^C\cup \mathscr{B}_{r,R}]
        \\
        &\ge 1-\bbP[\mathscr{T}_{\delta(R, \epsilon/2)}^C]- \bbP[\mathscr{B}_{r,R}]
        \\
        &\ge 1-\epsilon. \qedhere
    \end{align*}
    \end{proof}

    It is now straightforward to prove Corollary \ref{cor:FPPandBondPerc}. 
    First, assume that 
    \begin{align*}
        \inf_{x\in\bbZ^d}\liminf_{R\to\infty}\bbP[T_{K}(0, x)<T_K(0, \partial B_R)]=0,
    \end{align*} 
    and let $\epsilon>0$ be given. Fix $x_0\in\bbZ^d$ such that 
    \begin{align*}
        \liminf_{R\to\infty}\bbP[T_{K}(0, x_0)<T_K(0, \partial B_R)]<\epsilon/2
    \end{align*} and choose $r\ge 0$ such that $x_0\in B_r$. By Theorem \ref{thm:coupling}, there is an $R_0$ be such that for $R\ge R_0$,
    \begin{align*}
        \bbP[(K, R)\text{ log-uniform FPP contains IP on } B_r] \ge 1-\epsilon/2.
    \end{align*}

    Hence, for all $R\ge R_0$,
    \begin{align*}
        \bbP[\mathscr{I}(x_0)]\le \bbP[T_{K}(0, x_0)<T_K(0, \partial B_R)]+\epsilon/2
    \end{align*}
    and because $\liminf_{R\to\infty}\bbP[T_{K}(0, x_0)<T_K(0, \partial B_R)]<\epsilon/2$, we may find an $R_1\ge R_0$ such that $\bbP[T_{K}(0, x_0)<T_K(0, \partial B_{R_1})]<\epsilon/2$. Thus 
    \begin{align*}
        \bbP[\mathscr{I}(x_0)]\le \epsilon,
    \end{align*}
    and so 
    \begin{align*}
        \inf_{x\in\bbZ^d} \bbP[\mathscr{I}(x)]=0,
    \end{align*}
    giving, by \eqref{eqn:invasion_condition},
    \begin{align*}
        \theta_d(p_{c,d})=0. 
    \end{align*}

    Now assume that $\theta_d(p_{c,d})=0$, meaning $\inf_{x\in\bbZ^d}\bbP[\mathscr{I}(x)]=0$ by \eqref{eqn:invasion_condition}. Given $\epsilon>0$, fix $x_0\in\bbZ^d$ such that $\bbP[\mathscr{I}(x_0)]<\epsilon/2$ and note by Theorem \ref{thm:coupling}, for large $R$, 
    \begin{align*}
        \bbP[T_{K}(0, x_0)<T_K(0, \partial B_R)]&< \bbP[\mathscr{I}(x_0)] + \epsilon/2
        \\
        &<\epsilon.
    \end{align*}

    As this holds for all large $R$, $\liminf_{R\to\infty}\bbP[T_{K}(0, x_0)<T_K(0, \partial B_R)]<\epsilon$, and hence 
    \begin{align*}
        \inf_{x\in\bbZ^d}\liminf_{R\to\infty}\bbP[T_{K}(0, x)<T_K(0, \partial B_R)]=0.
    \end{align*}
    \section{Proofs of Lemmas} \label{sec:proofs_of_lemmas}
    
    In order to prove Lemma \ref{lem:Backbone}, we require the following lemma. 
    \begin{lemma}
    \label{lem:edgepassagetime}
    For any edge $e$ not adjacent to the origin,
    \begin{align*}
        T_K(0, e)=\min_{f\sim e}T_K(0, f)+\tau_{K, e}.
    \end{align*}
    \end{lemma}
    \begin{proof}
        Without loss of generality, for $v_1, v_2$ adjacent to $e$, neither of which are $0$, assume that $T_K(0, v_1)<T_K(0, v_2)$, and let $\gamma$ be the non-empty path from $0$ to $v_1$ for which $T_K(\gamma)=T_K(0, v_1)$. Because the passage time to $v_1$ is less than that of $v_2$, we know that $T_K(\gamma)\le T_K(\pi)$ for any path $\pi$ from $0$ to either $v_1$ or $v_2$. 

        There is a unique edge $e'\in \gamma$ adjacent to $e$, as if another was incident to $v_1$ then we could cut off a portion of the path contradicting the minimality of $\gamma$, and if another was incident to $v_2$, then $T_K(0, v_2)\le T_K(0, v_1)$. Let $v_1'$ be the vertex adjacent to $e'$ along with $v_1$. We claim that $T_K(0, e')=T_K(0, v_1')+\tau_{K, e'}$. Indeed as there is a path $\gamma\setminus\{e'\}$ from $0$ to $v_1'$ such that $T_K(\gamma\setminus\{e'\})< T_K(\gamma)$, we know that $T_K(0, v_1')<T_K(0, v_1)$. 

        We also claim that $T_K(0, v_1')=T_K(\gamma\setminus\{e'\})$. If there were to exist another path $\gamma'$ from $0$ to $v_1'$ with $T_K(\gamma')<T_K(\gamma\setminus\{e'\})$, then 
        \begin{align*}
            T_K(\gamma'\cup\{e'\})&=T_K(\gamma')+\tau_{K, e'}
            \\
            &< T_K(\gamma\setminus\{e'\})+\tau_{K, e'}
            \\
            &=T_K(\gamma).
        \end{align*}
        Because $\gamma'\cup\{e'\}$ forms a path from $0$ to $v_1$, this contradicts the minimality of $\gamma$. Thus we have that $T_K(0, v_1')=T_K(\gamma\setminus\{e'\})$

        Finally, we claim that $T_K(0, e)=T_K(0, e')+\tau_{K, e}$, and that $T_K(0, e')\le T_K(0, f)$ for all $f$ adjacent to $e$. Indeed, we know that
        \begin{align*}
            T_K(0, v_1)&=T_K(\gamma) 
            \\
            &= T_K(\gamma\setminus\{e'\}) + \tau_{K, e'} 
            \\
            &= T_K(0, v_1')+\tau_{K, e'} 
            \\
            &= T_K(0, e'),
        \end{align*} 
        and hence $T_K(0, e)=T_K(0, e')+\tau_{K, e}$. Moreover, if there exists some $f$ adjacent to $e$ such that $T_K(0, f)<T_K(0, e')$, then we know by an identical argument to the one above that a vertex $x\neq v_1, v_2$ which is adjacent to $f$ must be such that $T_K(0, f)=T_K(0, x)+\tau_{K, f}$. But then a path $\pi$ from $0$ to $x$ is such that $T_K(\pi\cup\{f\}) < T_K(\gamma)$, contradicting the minimality of $\gamma$ over all paths from $0$ to either $v_1$ or $v_2$. 
    \end{proof}
    
    \backbone*
    \begin{proof}
         We will prove this lemma by building up the orders $<_{\texttt{IP}}$ and $<_{\texttt{FPP}}$ from their minimal elements, and confirm that they are identical for all edges $e, e'<_{\texttt{IP}} v_R^{\texttt{IP}}$. Indeed, clearly if $e_{0}$ is the edge incident to $0$ of minimal weight, then $e_0$ is the minimum edge with respect to $<_{\texttt{IP}}$. It is also the minimum edge with respect to $<_{\texttt{FPP}}$, as for any vertex $x$ $T_K(0, x)\ge \min_{1\le k\le d} \{T_K(0, \pm v_k)\}$, where $v_k$ denote the standard basis vectors of $\bbR^d$. Proceeding via induction, assume that the two orderings are identical up to some edge $e_k$, and let $e_{k+1}$ be such that there is no $e'$ with $e_k<_{\texttt{IP}}e'<_{\texttt{IP}}e_{k+1}$. First, as the orderings are identical up to $e_k$, we have $e_k<_{\texttt{FPP}} e_{k+1}$ as well. Assume for the sake of contradiction that there is an $e'$ such that $e_k<_{\texttt{FPP}}e'<_{\texttt{FPP}} e_{k+1}$, and take $e'$ to be minimal in this respect. 
        
         First, $e'$ must have been adjacent to some edge $e'_{\texttt{adj}}\le_{\texttt{FPP}}e_k$. If $e'$ is adjacent to the origin, then we may take $e'_{\texttt{adj}}=e_0\le_{\texttt{FPP}}e_k$. If not, by Lemma \ref{lem:edgepassagetime} take $e'_{\texttt{adj}}$ adjacent to $e'$ such that 
         \begin{align*}
             T_K(0, e')=T_K(0, e'_{\texttt{adj}}) + \tau_{e', K}
         \end{align*}
         and thus $e'_{\texttt{adj}}<_{\texttt{FPP}}e'$, meaning $e'_{\texttt{adj}}\le_{\texttt{FPP}} e_k$ by the minimality of $e'$. 
         
        \begin{figure}[h!]
\[\begin{tikzcd}[arrows={line width=0.75pt, shorten >=-0.75ex, shorten <=-0.75ex}]
	\bullet & \bullet & \bullet & \bullet & \bullet \\
	\bullet & \bullet & \bullet & \bullet & \bullet \\
	\bullet & \bullet & {\raisebox{-0.15ex}{$\bullet_0$}} & \bullet & \bullet \\
	\bullet & \bullet & \bullet & \bullet & \bullet \\
	\bullet & \bullet & \bullet & \bullet & \bullet
	\arrow[no head, from=1-1, to=2-1]
	\arrow[no head, from=1-2, to=1-1]
	\arrow[no head, from=1-2, to=2-2]
	\arrow[no head, from=1-3, to=1-2]
	\arrow[no head, from=1-3, to=1-4]
	\arrow[no head, from=1-3, to=2-3]
	\arrow[no head, from=1-4, to=1-5]
	\arrow["{e'}", no head, from=1-4, to=2-4]
	\arrow[no head, from=1-5, to=2-5]
	\arrow[no head, from=2-1, to=2-2]
	\arrow[color={rgb,255:red,92;green,92;blue,214}, no head, from=2-2, to=2-3]
	\arrow["{e_{\texttt{adj}}'}"', color={rgb,255:red,214;green,92;blue,92}, no head, from=2-4, to=2-3]
	\arrow[no head, from=2-4, to=3-4]
	\arrow[no head, from=2-5, to=2-4]
	\arrow[no head, from=3-1, to=2-1]
	\arrow["{e_k}"', color={rgb,255:red,214;green,92;blue,92}, no head, from=3-1, to=3-2]
	\arrow[no head, from=3-1, to=4-1]
	\arrow[color={rgb,255:red,214;green,92;blue,92}, no head, from=3-2, to=2-2]
	\arrow[color={rgb,255:red,92;green,92;blue,214}, no head, from=3-2, to=4-2]
	\arrow[color={rgb,255:red,92;green,92;blue,214}, no head, from=3-3, to=2-3]
    \arrow["{e_{0}}", color={rgb,255:red,92;green,92;blue,214}, no head, from=3-3, to=3-2]
	\arrow["{e_{\texttt{adj}}^{(1)}}", color={rgb,255:red,214;green,92;blue,92}, no head, from=3-4, to=3-3]
	\arrow[no head, from=3-4, to=3-5]
	\arrow[no head, from=3-5, to=2-5]
	\arrow[color={rgb,255:red,214;green,92;blue,92}, no head, from=4-1, to=4-2]
	\arrow[no head, from=4-2, to=4-3]
	\arrow[no head, from=4-3, to=3-3]
	\arrow["{e_{\texttt{adj}}^{(0)}}"', color={rgb,255:red,214;green,92;blue,92}, no head, from=4-4, to=3-4]
	\arrow[no head, from=4-4, to=4-3]
	\arrow["{e_{k+1}}"', color={rgb,255:red,214;green,92;blue,92}, no head, from=4-4, to=4-5]
	\arrow[no head, from=4-5, to=3-5]
	\arrow[no head, from=5-1, to=4-1]
	\arrow[no head, from=5-2, to=4-2]
	\arrow[no head, from=5-2, to=5-1]
	\arrow[no head, from=5-3, to=4-3]
	\arrow[no head, from=5-3, to=5-2]
	\arrow[color={rgb,255:red,214;green,92;blue,92}, no head, from=5-4, to=4-4]
	\arrow[no head, from=5-4, to=5-3]
	\arrow[no head, from=5-4, to=5-5]
	\arrow[no head, from=5-5, to=4-5]
\end{tikzcd}\]
\caption{An example configuration in the coupled process in $\bbZ^2$. Edges $f$ drawn in red are such that $e'_{\texttt{adj}}\le f<_{\texttt{IP}}e'$, edges $g$ drawn in blue are such that $g<_{\texttt{IP}}e'_{\texttt{adj}}$. In this case, as $e^{(1)}_{\texttt{adj}}$ is adjacent to the origin and $e^{(1)}_{\texttt{adj}}>_{\texttt{IP}}e'_{\texttt{adj}}$, we stop our process at $e^{(2)}_{\texttt{adj}}=e_0$.}
\label{fig:edge_plot}
\end{figure}
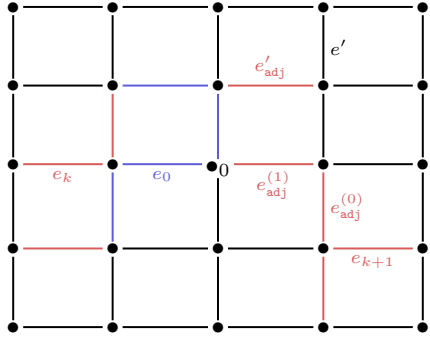

         By Lemma \ref{lem:edgepassagetime}, take $e_{\texttt{adj}}$ adjacent to $e_{k+1}$ such that $e_{\texttt{adj}}\le_{\texttt{IP}}e_k$ and $T_K(0, e_{k+1})=T_K(0, e_{\texttt{adj}})+\tau_{K,e_{k+1}}$ if $e_{k+1}$ is not adjacent to the origin, and $e_{\texttt{adj}}=e_0$ if it is. With $e^{(0)}_{\texttt{adj}}:=e_{\texttt{adj}}$, we may recursively find an edge $e^{(n+1)}_{\texttt{adj}}$ adjacent to $e^{(n)}_{\texttt{adj}}$ such that $e^{(n+1)}_{\texttt{adj}}<_{\texttt{IP}}e^{(n)}_{\texttt{adj}}$ and has the above property, stopping this procedure with the first (potentially 0) $N$ such that $e^{(N)}_{\texttt{adj}}\le_{\texttt{IP}} e'_{\texttt{adj}}$, see Figure \ref{fig:edge_plot}. We note $N+1< |E_R|$, as for $e_k, e_{k+1}<_{\texttt{IP}}v^{\texttt{IP}}_R$, all $e^{(i)}_{\texttt{adj}}$ are contained within $E_R$, and that there is no edge $e'_{\texttt{adj}}<_{\texttt{IP}}\Tilde{e}\le_{\texttt{IP}} e_{k+1}$ such that $w(\Tilde{e})>w(e')$, as $e'$ was adjacent to an already invaded edge and hence would have been chosen to be invaded before $\Tilde{e}$. 
         We let $e_{\max}$ be the edge of largest weight in $\{f: e'_{\texttt{adj}}<_{\texttt{IP}}f\le_{\texttt{IP}} e_{k+1}\}$, and note that by the above, $w(e')> w(e_{\max})$. 
         On the event $\mathscr{T}_{\delta}$, we must have that $w(e')\ge w(e_{\max})+\delta$, and hence 
         \begin{align*}
             \tau_{K,e'} = e^{K(R)w(e')}\ge e^{K(R)\delta}e^{K(R)w(e_{\max})} = |E_R| \tau_{K,e_{\max}}.
         \end{align*}
         Thus
         \begin{align*}
             T_K(0, e_{k+1}) &\le T_K(0, e^{(N)}_{\texttt{adj}}) + \tau_{K, e^{(N-1)}_{\texttt{adj}}} + \dots + \tau_{K, e^{(0)}_{\texttt{adj}}} + \tau_{K, e_{k+1}}
             \\
             &\le T_K(0, e'_{\texttt{adj}}) + (N+1)\tau_{K,e_{\max}}
             \\
             &\le T_K(0, e'_{\texttt{adj}}) + |E_R|\tau_{K,e_{\max}}
             \\
             &\le T_K(0, e'_{\texttt{adj}}) + \tau_{K,e'}
             \\
             &=T_K(0, e'),
         \end{align*}
         giving a contradiction to $e'<_{\texttt{FPP}}e_{k+1}$.
    \end{proof}
    
    \nolatesmallinvasions*
    \begin{proof} Assume to the contrary that there exists some $\epsilon>0$ and $r\ge0$ such that for any $R_0\in\bbN$ there exists an $R\ge R_0$ such that $\bbP[\mathscr{B}_{r,R}]\ge \epsilon$.

    Because $\mathscr{B}_{r,R+1}\subset \mathscr{B}_{r,R}$, we get that $\bbP[\mathscr{B}_{r,R+1}]\le \bbP[\mathscr{B}_{r, R}]$ for all $R$. Our assumption gives us that there exists a sequence of $R_k$'s such that $\bbP[\mathscr{B}_{r,R_k}]\ge \epsilon$, and hence we may conclude that $\bbP[\mathscr{B}_{r, R}]\ge \epsilon$ for all $R\ge R_0$.

    Let $\mathscr{B}_r = \cap_{R\ge R_0} \mathscr{B}_{r,R}$. Then because $\mathscr{B}_{r,R}$ is a decreasing sequence of events, 
    \begin{align*}
        \bbP[\mathscr{B}_r] = \lim_{R\to\infty} \bbP[\mathscr{B}_{r,R}]\ge \epsilon.
    \end{align*}
    We claim that $\mathscr{B}_r=\emptyset$. To see this, we consider any percolation configuration $\sigma\in\mathscr{B}_r$. We have that $\sigma\in\mathscr{B}_{r,R}$ for all $R\ge R_0$, and so there is a vertex in $B_r$ which is invaded after $v_R^c$ for all $R\ge R_0$. Because there are a finite number of vertices in $B_r$, we may let step $n\ge R_0$ be the step when the final vertex $v\in B_r$ is invaded for the configuration $\sigma$. But then $v_{n+1}^c$ could not have been invaded before $v$, contradicting $\sigma\in \mathscr{B}_{r, n+1}$. Therefore there can be no $\sigma\in\mathscr{B}_r$.  
    
    This gives us 
    \begin{align*}
        0=\bbP[\emptyset]=\bbP[\mathscr{B}_r]\ge \epsilon>0,
    \end{align*}
    a contradiction.
    \end{proof}

    \section{Simulation of First Passage Percolation}
    In this section we provide some simulations of the behavior of log-uniform distributed first passage percolation in $\bbZ^2$ for $K(R, 0.01)$ as defined in Theorem \ref{thm:coupling}, with the code used to run these simulations contained within this Github repository: \href{https://github.com/aldomorelli/UNM-Percolation-Research-2025}{Log Uniform First Passage Percolation Simulation}. Trials of the first passage process were ran until time $t=T_{K}(0, \partial B_R)$, and for each $x\in B_R$ the proportion of trials where the event $T_K(0, x)< T_K(0, \partial B_R)$ occurred was recorded. A plot of these proportions for $R=100$ is included in Figure \ref{fig:propinfectedplot}. The level sets of values away from $0$ appear to be circular although the boundary being compared to was the $\ell_1$ ball.

\begin{figure}[htbp]
    \centering
    \begin{subfigure}[t]{0.45\textwidth}
        \centering
        \includegraphics[width=\linewidth]{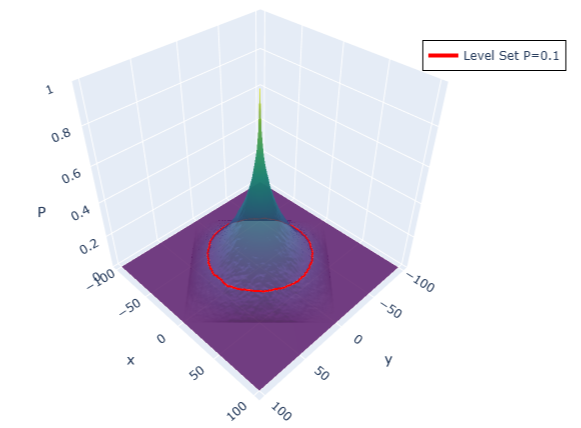}
        \caption{Boundary $\partial B_{100}$}
        \label{fig:propinfectedplot}
    \end{subfigure}
    \quad
    \begin{subfigure}[t]{0.45\textwidth}
        \centering
        \includegraphics[width=\linewidth]{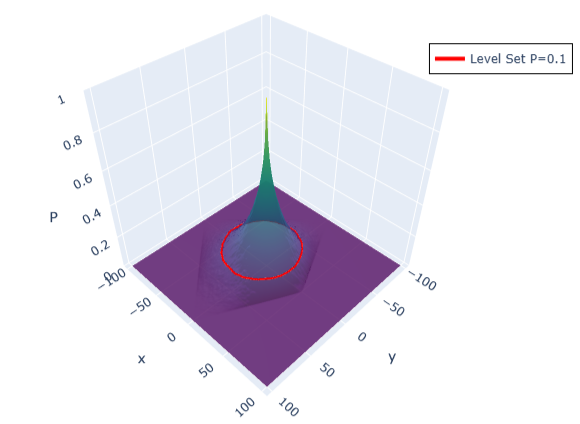}
        \caption{\parbox[t]{.9\linewidth}{Boundary $-x+|y|=100$ for $x\le 0$, \\
        $2x+|y|=100$ for $x>0$}}
        \label{fig:lopsidedplot}
    \end{subfigure}
    \caption{Simulations recording the number $P$ of trials $(n=10000)$ where the vertex with coordinates $(x,y)$ had passage time less than the passage time to the given boundary. The level curve $P=0.1$ is given in red. Outside of the boundary, the proportions are $0$ as expected, and the level curves away from $0$ of both distributions appear to be circular.}
\end{figure}

  Simulations were performed with lopsided boundary conditions appear to show similar results, as demonstrated in Figure \ref{fig:lopsidedplot}. We think such results may be interesting in connection to a limiting shape in first passage percolation, as the coupling with invasion percolation Theorem \ref{thm:coupling} might be utilized.

    Restricting to the slice $y=0$, Figure \ref{fig:slicecomparison} shows the proportions of vertices $(x,0)$ with passage times less than $T_{K}(0, \partial B_R)$ for $R=100,200, 500, 1000$, where $[-R,R]$ has been scaled to $[-1,1]$.
    The resulting curve resembles a curve of the form $1-|x|^{\alpha(R)}$ for some $0<\alpha(R)<1$. 
    A regression of this form for $R=1000$ is given in Figure \ref{fig:regression}, which was accomplished by finding the least squares solution $\alpha$ to the system $\alpha \log|x|=\log(1-P)$. We note that as a consequence of Theorem \ref{thm:coupling}, because not every vertex has probability $1$ of being invaded, if such an $\alpha(R)$ exists, then $\alpha(R)\to 0$ as $R\to\infty$. If the overall distribution is approximately radially symmetric regardless of boundary, this restriction to $y=0$ should be fairly representative of all directions. 

    \begin{figure}[h!] 
\centering
\begin{subfigure}{.45\textwidth}
  \centering
  \includegraphics[width=\linewidth]{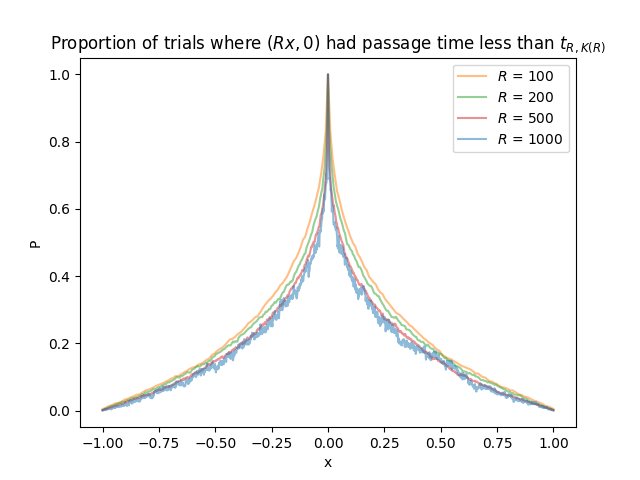}
  \phantomsubcaption
  \label{fig:slicecomparison}
\end{subfigure}
\quad
\begin{subfigure}{.45\textwidth}
  \centering
  \includegraphics[width=\linewidth]{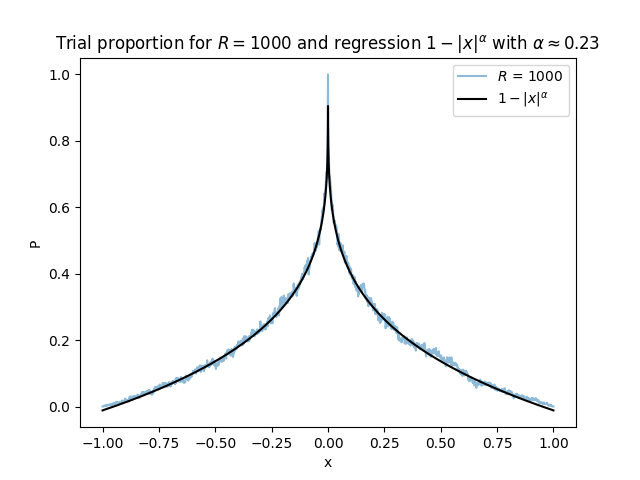}
  \phantomsubcaption
  \label{fig:regression}
\end{subfigure}
\caption{On the left, figure \ref{fig:slicecomparison} records the proportion of trials for which the vertices $(Rx,0)$ for $x=k/R$ with $k\in \{-R, \dots, R\}$ had passage time at most $T_{K}(0, \partial B_R)$ for $R=100,200, 500,1000$. We see that the shapes resemble that of $1-|x|^{\alpha(R)}$ for some exponent $\alpha(R)$. On the right, figure \ref{fig:regression} shows in blue is the proportion of trials for which the vertices $(1000x,0)$ for $x=k/1000$ with $k\in \{-1000, \dots, 1000\}$ had passage time at most $T_{K}(0, \partial B_{1000})$. In black is a regression of the form $1-|x|^{\alpha}$ with $\alpha\approx 0.23$ and correlation coefficient $r=0.998$.}
\end{figure}

    \section*{Acknowledgements}
    I would like to thank Dr.\ Sarah Percival for her guidance and support throughout this work and Dr.\ Matthew Junge for his comments on a draft of this manuscript. I am also grateful to Dr.\ Charles Newman, and Joshua Meisel for their helpful discussions. I would like to thank the UNM Center for Advanced Research Computing, supported in part by the National Science Foundation, for providing the research computing resources used in this work. Research reported in this publication was supported by the National Cancer Institute of the National Institutes of Health under Award Number U54CA272167. The content is solely the responsibility of the author and does not necessarily represent the official views of the National Institutes of Health.
    
    \bibliographystyle{alpha}
    \bibliography{references}

\newcommand{\etalchar}[1]{$^{#1}$}
\begin{thebibliography}{CKLW82}

\bibitem[ADH16]{ADH16}
A.~Auffinger, M.~Damron, and J.~Hanson.
\newblock 50 years of first passage percolation, 2016.

\bibitem[BH57]{BH57}
S.~Broadbent and J.~M. Hammersley.
\newblock Percolation processes.
\newblock {\em Mathematical Proceedings of the Cambridge Philosophical Society}, 53:629 -- 641, 1957.

\bibitem[CCN87]{CCN87}
J.~T. Chayes, L.~Chayes, and C.~M. Newman.
\newblock Bernoulli percolation above threshold: An invasion percolation analysis.
\newblock {\em The Annals of Probability}, 15(4):1272--1287, 1987.

\bibitem[CJLP24]{CJLP24}
S.~Chiaradonna, P.~Jevti{\'c}, N.~Lanchier, and S.~Pesic.
\newblock Framework for cyber risk loss distribution of client-server networks: A bond percolation model and industry specific case studies.
\newblock {\em Applied Stochastic Models in Business and Industry}, 40(6):1712--1733, November 2024.

\bibitem[CKLW82]{CKLW82}
R.~Chandler, J.~Koplik, K.~Lerman, and J.F. Willemsen.
\newblock Capillary displacement and percolation in porous media.
\newblock {\em Journal of Fluid Mechanics}, 119:249–267, 1982.

\bibitem[DC18]{DC18}
H.~Duminil-Copin.
\newblock Introduction to {B}ernoulli percolation, October 2018.

\bibitem[DLW15]{DLW15}
M.~Damron, W.~Lam, and X.~Wang.
\newblock Asymptotics for {$2D$} critical first passage percolation, 2015.

\bibitem[Ear22]{E22}
Mike Earnest.
\newblock Average minimum distance between $n$ points generate i.i.d. with uniform dist.
\newblock Mathematics Stack Exchange, 2022.
\newblock URL:https://math.stackexchange.com/q/2001026 (version: 2022-09-13).

\bibitem[Mey07]{ML07}
L.~Meyers.
\newblock Contact network epidemiology: Bond percolation applied to infectious disease prediction and control.
\newblock {\em Bulletin (New Series) of the American Mathematical Society}, 44, 02 2007.

\bibitem[WW83]{WW83}
D.~Wilkinson and J.~F. Willemsen.
\newblock Invasion percolation: a new form of percolation theory.
\newblock {\em Journal of Physics A: Mathematical and General}, 16(14):3365, oct 1983.

\bibitem[ZYH{\etalchar{+}}24]{ZYH24}
X.~G. Zheng, I.~Yamauchi, M.~Hagihala, E.~Nishibori, T.~Kawae, I.~Watanabe, T.~Uchiyama, Y.~Chen, and C.~N. Xu.
\newblock Unique magnetic transition process demonstrating the effectiveness of bond percolation theory in a quantum magnet.
\newblock {\em Nature Communications}, 15(1):9989, Nov 2024.

\end{thebibliography}
\end{document}